\documentclass[11pt,reqno]{amsart}
\usepackage[top=1in, bottom=1.25in, left=1.25in, right=1.25in]{geometry}
\usepackage{amsmath,amssymb}
\usepackage[T1]{fontenc}
\usepackage{mathrsfs}
\usepackage{graphics}
\usepackage{tikz}
\usetikzlibrary{patterns}
\usepackage{color}
\usepackage{caption}
\usepackage{subfigure}
\usepackage{latexsym}
\usepackage{graphicx}
\usepackage{float}
\usepackage{extarrows}
\usepackage{epstopdf}
\usepackage{appendix}

\newtheorem{satz}{Proposition}[section]
\newtheorem{lem}[satz]{Lemma} 
\newtheorem{remark}[satz]{Remark}
\newtheorem{thm}[satz]{Theorem}

\newtheorem{cor}[satz]{Corollary}

\definecolor{gray}{gray}{0.50}
\definecolor{lred}{rgb}{1.0,0.5,0.5}
\definecolor{dgreen}{rgb}{0,1,1}

\numberwithin{equation}{section}

\newcommand{\chookrightarrow}{\mathrel{\lhook\joinrel\relbar\kern-.8ex\joinrel\lhook\joinrel\rightarrow}}

\newcommand{\R}{\mathbb{R}}

\newcommand{\N}{\mathbb{N}}

\newcommand{\dd}{\mathrm{d}}
\newcommand{\e}{\varepsilon}

\newcommand{\beq}{\begin{equation}}
\newcommand{\eeq}{\end{equation}}

\normalsize
\definecolor{luh-dark-blue}{rgb}{0.0, 0.313, 0.608}
\definecolor{lred}{rgb}{1.0,0.5,0.5}

\usepackage[colorlinks=true]{hyperref}
\hypersetup{urlcolor=luh-dark-blue, citecolor=lred, linkcolor=luh-dark-blue}

\begin{document}
\title[Symmetry and decay of traveling waves]{Exponential decay and symmetry of solitary waves to  Degasperis-Procesi equation}
\author{Long Pei}
\address{Department of Mathematical Sciences, Norwegian University of Science and Technolgy, 7491 Trondheim, Norway.}
\email{longp@kth.se}


\begin{abstract}
We improve the decay argument by [Bona and Li, J. Math. Pures Appl., 1997] for solitary waves of general dispersive equations and illustrate it in the proof for  the exponential decay of solitary waves to steady Degasperis-Procesi equation in the nonlocal formulation. 
In addition, we give a method which confirms the symmetry of solitary  waves, including those of the maximum height.  Finally, we discover how the symmetric structure is connected to the steady structure of solutions to the Degasperis-Procesi equation, and give a more intuitive proof for symmetric solutions to be traveling waves.  The improved argument and new methods above can be used for the decay rate of solitary waves to many other dispersive equations and will give new perspectives on  symmetric solutions for general evolution equations.
\end{abstract}
\maketitle

\section{Introduction}

The Degasperis-Procesi (DP) equation
\begin{equation}\label{eq: DP local}
u_t-u_{xxt}+4uu_x-3u_{x}u_{xx}-uu_{xxx}=0
\end{equation}
is a unidirectional model for shallow water waves (see \cite{degasperis1999asymptotic}) and can be reformulated as a nonlocal equation
\begin{equation}\label{eq:DP nonlocal}
\partial_{t}u+u\partial_{x}u+\partial_{x}L(\frac{3}{2}u^2)=0,
\end{equation}
where the dispersive operator $L=(1-\partial_{x}^{2})^{-1}$ corresponds to the Fourier symbol $m(\xi)=(1+\xi^2)^{-1}$ and a convolution kernel function $K(x)=\frac{1}{2}e^{-|x|}$. 
Being completely integrable and having bi-hamiltonian structure \cite{degasperis1999asymptotic}, this equation together with KdV and Camassa-Holm were three well-known representatives in both integrable system theory and water wave problems. Although firstly put forward from the perspective of integrability, this model was later rigorously derived as a model for shallow water waves and proved to have the same accuracy as the Camassa-Holm equation \cite{constantin2009hydrodynamical}. The Degasperis-Procesi equation is locally well-posed in the classical Sobolev space $H^{s}$, $s>\frac{3}{2}$, in both periodic and non-periodic settings \cite{yin2003cauchy}, and it allows global weak and classical solutions \cite{yin2004global,yin2003global} while the latter may blow up in the form of wave-breaking \cite{escher2006global}. Soliton solutions of Degasperis-Procesi equation can be found by inverse scattering technique \cite {degasperis2002new,constantin2010inverse}.  Later,
traveling wave solutions (both periodic and solitary)   to \eqref{eq: DP local} were found in \cite{vakhnenko2004periodic}, and  Lenells classified in \cite{lenells2005traveling} all possible traveling wave solutions, which include smooth waves, peaked waves cusped waves, stumped waves and their reasonable compsition. Very recently, Arnesen \cite{arnesen2018non} worked on the non-local formulation \eqref{eq:DP nonlocal} and proved that differentiable, symmetric traveling solutions with uniform bound have the wave speed as the upper bound and are smooth when wave height is strictly smaller than wave speed $c$. In addition,  crests of periodic  waves will turn to peaks when the wave height reachs the wave speed.

\medskip

This papers focuses on solitary waves (steady solutions with decay at infinity) of Degasperis-Procesi equation and the motivation comes from several aspects: Firstly, Bona and Li studied in \cite{Bona1997decay} the decay and analyticity of solitary waves to a class of evolution equations in the steady form 
\begin{equation}
f=k*G(f)
\end{equation}
where $k$ denotes the convolution kernel function, and $G(\cdot)$ is locally bounded and has superlinear growth.  The procedure of proving exponential decay of solitary waves mainly involves two  steps: step 1 for algebraic decay in some $L^{p}(\mathbb{R})$ spaces and  step 2 for a delicate control of $L^{1}$ norm of $|x|^{n}\phi(x)$, $n\in \mathbb{N}$, to guarantee exponential decay. We hope to  simplify this two-step procedure. In fact, the following polynomial type convolution estimate (see also other similiar estimates in \cite[Lemma 3.1.1]{Bona1997decay})
	\beq\label{eq:polynomial estimate bona and li}
\int_{0}^{\infty}\frac{|x|^{l}}{(1+\epsilon |x|)^{m}(1+|y-x|)^{m}}\dd x\leq B \frac{|y|^{l}}{(1+\epsilon |y|)^{m}}, \quad |y|\geq 1,
\eeq
is the key for the algebraic decay in \cite{Bona1997decay}. We  improve this polynomial type estimate to exponential type estimate. In this way,  the algebraic decay estimate of solitary waves can be skipped in the argument by Bona and Li and we can prove the exponential decay.  In view that  the improved exponential type estimate (see Lemma \ref{lemma:key estimate for solution decay} below) are independent of the form of dispersive equations, it is expected to simplify the proof for exponential decay of solitary waves for more general dispersive equations  as 
\eqref{eq:polynomial estimate bona and li} does for algebraic decay.
\medskip

The second aspect for motivation is related to symmetry  issues of the highest solitary wave to nonlinear dispersive equations. Traveling waves solutions are often studied by  \emph{a priori} assuming that they are  even or symmetric, and it rises the question  whether there exists asymmetric traveling waves. For dispersive equations where complete integrity is unknown, the inverse scattering technique for obtaining exact solutions will fail. In this case, the symmetry of solutions  are often obtained by the classical method of moving planes  put forward by Aleksandrov \cite{Alk} and Serrin \cite{Serr} (see also \cite{Craig} about this method for water waves). However, two obstacles will appear when applying the method of moving planes: one is to remove the \emph{a priori} monotonicity condition on solitary waves (essentially, this condition assumes that the wave has only one crest); the other is to prove the symmetry for waves of the maximum height. These difficulities can be well-illustrated by the symmetry problem of supercritical solitary wave solutions to the steady Whitham equation (see \cite{bruell2017symmetry}) 
\begin{equation}\label{eq:steady whitham}
\phi(c-\phi)=K_{w}*\phi^2,
\end{equation}
where $K_{w}$ denotes the kernel function for the Whitam equation. The monotonicity condition on solitary waves was removed by using the exponential decay of solitary waves and inspired by the idea in \cite{CLO} for integral equations induced from fractional Laplacian. However, when the solitary wave $\phi$ reaches the maximum height $\frac{c}{2}$ at the crest, the left side of \eqref{eq:steady whitham}  will generate a factor  $c-\phi(x)-\phi(2\lambda-x)$ for some $\lambda\in\R$. This factor approaches $0$ as $x$ approaches $\lambda$ and causes singularity when it is moved to the right side of the equation. In this case the argument for symmetry in \cite{bruell2017symmetric} fails to give a contradiction. This obstacle also appears for DP equation in the nonlocal form when $\phi$ reaches the maximum height $c$. In this paper, we  get around this obstracle by studying the local structure of the solitary wave near the crest $\phi=c$, and then manage to prove the symmetry also for the highest wave. This new idea is expected to work after modification for the symmetry of the highest solitary wave to the Whtiham and other  dispersive equations. 

\medskip

The third aspect of motivation comes from the classification of symmetric solutions to general evolution equations. In \cite{EHR}, the authors put forward a principle for a class of equation for which  solutions with \emph{a priori} spatial symmetry must be traveling waves. This principle was later extended to cover nonlocal equations and differential systems in \cite{bruell2017symmetric}, where two new principles were also found. The Degasperis-Procesi equation satisfies the principle in \cite{EHR} so that symmetric solutions must be traveling waves. The beautiful proof in \cite{EHR}, however, is quite constructive and does not give further information about how symmetric structure is related to the steady structure of those waves. In this paper, we study the two restriction conditions that symmetric solutions satisfy and find that each of them determines one aspect of the steady structure of these solutions: the fixed shape of wave profile and the constant propagation speed.  In this way, we give a more intuitive, straightforward proof for symmetric solutions to be traveling waves. This idea can be used for a family of equations whose structure satisfies Principle P1 in \cite{EHR,bruell2017symmetric}, including  KdV and Benjamin-Ono equation. 

\medskip

The final aspect of motivation comes from the classification of solitary waves to the Degasperis-Procesi equation. Inserting the ansatz $u(t,x)=\phi(x-c t)$ for traveling wave solutions into \eqref{eq: DP local}, one obtains the  Degasperis-Procesi equation in steady form with some integration parameter $a$.   According to the value of $a$, all possible  traveling wave solutions, periodic or solitary, were completely classified by Lenells in  \cite{lenells2005traveling}, including smooth waves, peaked waves, cusped waves, stumped waves and their proper composition. In this paper, we work on  DP equation in the nonlocal form \eqref{eq:DP nonlocal}  and get the following steady equation 
\begin{equation}\label{eq:steady DP with a}
\frac{\phi}{3}(2c-\phi)=L\phi^2+a,
\end{equation}
where $a $ denotes the integration constant. Unlike the Whitham equation and many others, it is not possible to use Galiean transformation to remove the constant $a$ in  \eqref{eq:steady DP with a}. However, we  prove that the constant $a$ must be trivially $0$ for solitary waves with decay (meaning that $\phi(x)\to 0$ in  as $|x|\to \infty$) so that these waves actually solves the steady equation 
\begin{equation}\label{eq:steady DP}
\frac{\phi}{3}(2c-\phi)=L\phi^2.
\end{equation}
In addition, we prove that that these waves are symmetric with respect to the only crest at some point and are strictly monotone on each side of the crest. Therefore, a solitary solution $\phi$ with decay only has one crest at a single point, excluding stumped solutions in \cite{lenells2005traveling} and the possibility to compose solitary waves with different propagation speeds into new solitary waves\footnote{This is because solitary waves with different propagation speeds will seperate from each other during later propagation so that their composition will have more than one crest and could not be a solitary solution to the steady equation \eqref{eq:steady DP}.}. Moreover, the peaked waves defined and found in \cite{lenells2005traveling} is only locally symmetric at the peak of a solitary wave, so our result improves this local symmetry near the peak  to global symmetry for the whole solitary wave. It is worth to point out that these findings do not contradict with the fact that the Degasperis-Procesi allows for  multipeakon solutions \cite{constantin2017dressing,lundmark2003multi}, which are not steady solutions.

\medskip

We now state the structure of this paper. Section \ref{sect:xponential decay} starts with an estimate where the kernel $K(\cdot)$ is convoluted with exponential type functions. Based on this lemma, we prove that the solitary solutions decay exponentially fast at infinity and the decay rate is at  least as good as the decay rate of the kernel $K(\cdot)$. Section \ref{sect:symmetry and one crest} focuses on the symmetry of solitary waves. In particular, we prove symmetry for solitary waves with height smaller than the wave speed in section \ref{sect: symmetry below maximal height}, while waves with the maximum height are treated in section \ref{sect: symmetry for maximal height}. Finally, we give a new proof in section \ref{sect: symmetry to traveling} for the classification principle that classical symmetric solutions to the Degasperis-Procesi equation must be traveling wave solutions\footnote{Such classification principle could also be formulated similarly in the weak setting with distribution theory, see \cite{EHR,annageyer2015}, but it is not our focus here. }.

\section{Exponential decay of solitary waves at infinity}\label{sect:xponential decay}
\noindent  
For a traveling wave solution $u(t,x)=\phi(x-ct)$ with  speed $c$, the sign of $c$ distinguishes only the direction of the propagation of the wave. So, we will only work with $c>0$ in the following. As mentioned above, direct calculation by fourier analysis gives that 
\beq\label{eq:def L}
\mathscr{F}[Lf](\xi)=\frac{1}{1+\xi^2}\mathscr{F}[f](\xi)=\mathscr{F}[K*f](\xi)
\eeq
where $\mathscr{F}$ denotes the usual Fourier transform and $K(x)=\frac{1}{2}e^{-|x|}$ denotes the convolution kernel of $L$. By definition,  the operator $L$ lifts a $L^{\infty}-$bounded function to a continuous function (see \cite{EW,arnesen2018non} for details), we will then work with continuous solutions in the following. We also need some elementary concepts from topology (see \cite{PeterZisman2012} for details). A pointed space is a topological space with a distinguished point called \emph{basepoint}. A  map $g$ from a pointed space $(X,x_0)$ to another pointed space $(Y,y_0)$   is a \emph{homomorphism} if $g$ is a continuous map from $X$ to $Y$ and preserves the basepoints, namely $g(x_0)=y_0$. In particular, we call $g$ a homomorphism on $(X,x_0)$ if it is a homomorphism from $(X,x_0)$ to itself.  We can choose the origin as basepoint so that $(\R,0)$ forms a pointed space with  the usual  Euclidean metric topology. We start with the proof for $a$ to vanish for  solitary waves to steady Degasperis-Procesi equation \eqref{eq:steady DP with a}, which follows directly from the the lemma below for the structure of general convolution equations.
 \begin{lem}\label{lem: remove constant}
Let  $G$ be a homomorphism from the pointed space $(\R,0)$. Let $k\in L^{1}(\R)$  decay at infinity and $H$ be a continuous function on $\R$.  If the following convolution equation 
\begin{equation}\label{eq: general to remove constant}
f=k*G(f)+H(f)
\end{equation}
 has a solution $f(x)$ which is continuous and decays at infinity. Then, $H$ is a homomorphism on  $(\R,0)$.
 \end{lem}
 \begin{proof}
It suffices to prove that $H$ preserves the origin as basepoints, i.e., $H(0)=0$. Since $f(x)$ decays at infinity, we only need to prove that $k*G(f)$ vanishes as $|x|\to\infty$ in \eqref{eq: general to remove constant}. Note that 
 	\beq\label{eq: a vanish}
 k*G(f)(x)=\int_{|x-y|<N}k(y)[G(f)](x-y)\dd y+\int_{|x-y|>N}k(y)[G(f)](x-y)\dd y
 	\eeq
 	for some $N\in \R$.
For any small $\eta>0$, we can choose $N$ large enough such that $|f(x)|<\frac{\eta}{2}$ for all $|x|>N$. Then, we have
 	\beq\label{eq: a vanish term 2}
 	\left|\int_{|x-y|>N}k(y)[G(f)](x-y)\dd y\right|<\left|\sup_{|f|<\frac{\eta}{2}}G(f)\right| \int_{|x-y|>N}|k(y)|\dd y \leq\left|\sup_{|f|<\frac{\eta}{2}}G(f)\right| \|k\|_{L^{1}(\R)}
 	\eeq
Note that  $k$ decays at infinity, so we can fix the above $N$ and $\eta$, and choose $M_1>0$ large enough such that $k(y)<\frac{\eta}{8N G(\|f\|_{L^{\infty}(\R)})}$ for all $|y|>M_1$.  Then, for any $y$ such that $|x-y|<N$ and   $|x|>M_1+N$, we have 
 	\[
 	M_1<|x|-N<|y|<|x|+N
 	\]
Therefore, for $|x|>M_1+N$, we have
 	\beq\label{eq: a vanish term 1}
 \left|\int_{|x-y|<N}k(y)[G(f)](x-y)\dd y\right|<2N G(\|f\|_{L^{\infty}(\R)})\sup_{|y|>M_1}k(y)<\frac{\eta}{4}
 	\eeq
Now, for any small $\epsilon>0$, we can choose $\eta<\epsilon$ sufficient small so that  $\left|\sup_{|f|<\frac{\eta}{2}}G(f)\right|<\frac{\epsilon}{4\|k\|_{L^{1}(\R)}}$ due to the fact that $G$ is a homomorphism on $(\R,0)$.  Then, we insert \eqref{eq: a vanish}, \eqref{eq: a vanish term 2}, \eqref{eq: a vanish term 1} into  \eqref{eq: general to remove constant}, and get
\begin{equation}
|H(f)(x)|\leq |f(x)|+|k*G(f)(x)|<\epsilon
\end{equation}
for all $|x|>M_1+N$.  The lemma then follows directly from the decay of $f$ at infinity and the continuity of $H$.
\end{proof}

\begin{remark}
	The Galilean transform as a usual trick to remove integration constants fails here. The idea in the Lemma \ref{lem: remove constant} is expected to work for more general settings where the kernel function is integrable and has decay at infinity, such as the Whitham equation \cite{bruell2017symmetric}. 
\end{remark}

As a direct consequence of Lemma \ref{lem: remove constant}, we have the following corollary for the integration constant $a$ to be trivially $0$.

\begin{cor}
The integration constant $a$ in \eqref{eq:steady DP with a} vanishes for continuous solitary waves with decay.
\end{cor}
\begin{proof}
By using Lemma \ref{lem: remove constant} with $G(\phi)=\phi^2$ and $H(f)(x)=a$, we see that $\displaystyle\lim_{x\to \infty}H(f)(x)=0$ which implies $a=0$. 	
	\end{proof}

To proceed, we first give the lower and upper bound of solitary waves.

\begin{lem}\label{lem: bound for solitary waves}
Nontrivial continuous solitary waves with decay to \eqref{eq:steady DP} satisfies
\beq\label{eq:phi bound}
0<\phi\leq\sup_{x\in\R}\phi<2c.
\eeq
\end{lem}
\begin{proof}
The strict positiveness of $K(x)$ implies that $L$ is a strictly monotone operator on continuous bounded functions, i.e., $Lf>Lg$ if $f\geq g$ but $f\neq g$. In addition, straight calculation shows that $LC=C$ for any constant $C$. Therefore, we derive from \eqref{eq:DP nonlocal} that 
\beq
\phi^2-2c\phi=-3L\phi^2<0,
\eeq
which implies that $\phi\in (0,2c)$. The decay of $\phi$ indicates that  $\sup_{x\in\R}\phi$ must be reached at some finite $x_0\in\R$ so that \eqref{eq:phi bound} follows.
\end{proof}

\begin{remark}
A recent work \cite{arnesen2018non} by Arnesen shows that all $L^{\infty}$-bounded traveling waves has wave speed $c$ as upper bound. However, we do not need this better upper bound for the estimate of decay rate of solitary waves. 
\end{remark}

\medskip

With the above bound for solitary waves ready, the decay argument by Bona-Li in \cite{Bona1997decay} could be used to prove the exponential decay for solitary waves. To proceed, we recall the two-step procedure by Bona and Li: firstly derive algebraic decay of solitary waves; then improve the algebraic decay to exponential decay by making delicate control of some $L^{p}(\R)$ norm of solitary waves with monomial weight $|x|^{n}$ for each $n\in \N$. The key in the algebraic decay is a convolution estimate for functions of polynomial type. In fact, Let  $F_{1}(x)$ and $F_{2}(x)$ be given by $F_{1}(x):=\frac{|x|^{l}}{(1+\sigma |x|)^{m}}$ and $F_{2}(x):=(1+ |x|)^{-m}$. Then, it is proved essentially by Bona and Li that
\begin{equation}\label{eq:intuitive explanation}
F_{1}*F_{2}(x)\lesssim F_{1}(x)
\end{equation}
where $\lesssim$ means $\leq$ up to some constant relying on the indices $l$ and $m$. Intuitively, this statement claims that the convolution of two polynomial functions of negative order  could be controlled by the one with higher order. We find that this philosophy also holds if polynomials are replaced by exponential functions in proper formulation. In particular, let $G_{1}(x):=\frac{e^{l|x|}}{(1+\sigma e^{|x|})^{m}}$ and $G_{2}(x):=e^{-m|x|}$. Then it is true that
\begin{equation}\label{eq:intuitive explanation with exponential}
G_{1}*G_{2}(x)\lesssim G_{1}(x).
\end{equation}
With this new estimate \eqref{eq:intuitive explanation with exponential} for exponential functions, neither the algebraic decay of solitary waves nor the delicate control of the $L^{1}(\R)$ norm of $|x|^{n}\phi$ for each $n\in\N$ is needed, while the  exponential decay of solitary waves could be directly obtained. In this way, the  proof for exponential decay can be considerably simplified. We formulate the new estimate for exponential functions in the following Lemma.

\begin{lem}[Convolution estimate of exponential type]\label{lemma:key estimate for solution decay}
	For $0<l<m$ and any $\sigma>0$, the following inequatlity holds
	\beq\label{eq:key estimate}
	\int_{\R}\frac{e^{l|x|}}{(1+\sigma e^{|x|})^{m}e^{m|x-y|}}\dd x\leq B \frac{e^{l|y|}}{(1+\sigma e^{|y|})^{m}}, \quad y\in\R,
	\eeq
	where $B=(\mathrm{min}\{l,m-l\})^{-1}$.
\end{lem}
\begin{proof}
By symmetry of the structure in \eqref{eq:key estimate}, it suffices to prove for the case $y>0$.  Note that
\begin{equation*}
\int_{0}^{\infty}\frac{e^{l|x|}}{(1+\sigma e^{|x|})^{m}e^{m|x-y|}}\dd x=\left(\int_{0}^{y}+\int_{y}^{\infty}\right)\frac{e^{lx}}{(1+\sigma e^{x})^{m}e^{m|x-y|}}\dd x=:I_{1}+I_{2}
\end{equation*}
	For $I_{1}$, we have
	\begin{equation*}
	\begin{split}
	I_{1}=\int_{0}^{y}\frac{e^{lx}}{(1+\sigma e^{x})^{m}e^{m(y-x)}}\dd x\leq \frac{e^{ly}-1}{ e^{my}(\sigma+e^{-y})^{m}l}
	\leq\frac{e^{ly}}{l(1+\sigma  e^{y})^{m}}
	\end{split}
	\end{equation*}
	For $I_{2}$, we have
	\begin{equation*}
	\begin{split}
	I_{2}=\int_{y}^{\infty}\frac{e^{lx}}{(1+\sigma e^{x})^{m}e^{m(x-y)}}\dd x\leq \frac{e^{my}}{(1+\sigma e^{y})^{m}} \int_{y}^{\infty}e^{(l-m)x}\dd x\leq \frac{(m-l)^{-1}e^{ly}}{(1+\sigma e^{y})^{m}}
	\end{split}
	\end{equation*}
On the other hand, we have
\begin{equation*}
	\int_{-\infty}^{0}\frac{e^{l|x|}}{(1+\sigma e^{|x|})^{m}e^{m|x-y|}}\dd x=\left(\int_{0}^{y}+\int_{y}^{\infty}\right)\frac{e^{lx}}{(1+\sigma e^{x})^{m}e^{m(y+x)}}\dd x=:I_{3}+I_{4}
\end{equation*}
For $I_3$, we have
\begin{equation}
I_3\leq  \frac{e^{-my}}{(\sigma+e^{-y})^m}\frac{1}{2m-l}(1-e^{(l-2m)y})<\frac{e^{ly}}{(\sigma e^{y}+1)^m}\frac{1}{2m-l},
\end{equation}
where in the last inequality we used the fact $0<1-e^{(l-2m)y}<e^{ly}$.
For $I_4$, we have
\begin{equation}
I_4\leq  \frac{e^{-my}}{(\sigma e^{y}+1)^m}\frac{1}{m-l}e^{(l-m)y}<\frac{e^{ly}}{(\sigma e^{y}+1)^m}\frac{1}{m-l}.
\end{equation}
	The inequality \eqref{eq:key estimate} and hence this lemma follow directly.
\end{proof}

We now illustrate how the estimate of exponential type could be used to prove directly the exponential decay of solitary solutions $\phi$ to the Degasperis-Procesi equation \eqref{eq:DP nonlocal}. For convenience, we introduce the notation $M:=\sup_{x\in\R}\phi$.
\begin{thm}[Exponential decay of solitary waves]\label{eq:solution exponential decay}
The image of the map $x\mapsto e^{|x|} \phi(x)$ is a bounded, simply connected set in $[0,\infty)$.
\end{thm}

\begin{proof}
We first prove that 
\begin{equation}\label{eq:expo decay in Lp}
e^{\alpha|\cdot|} \phi(\cdot)\in L^{q}(\R)
\end{equation}
for any $\alpha\in(0,1)$ and $q>1$.
Since $e^{\alpha|x|}K(x)\in L^{p}(\R)$ for any $\alpha\in(0,1)$ and $p>0$, we can introduce a constant $C_{\alpha,p}$ given by
	\begin{equation*}
		C_{\alpha,p}:=3(2c-M)^{-1}\|e^{\alpha|\cdot|}K(\cdot)\|_{L_{p}(\R)},
	\end{equation*} 
where  $p$ is chosen to be the conjugate of $q$, i.e., $\frac{1}{p}+\frac{1}{q}=1$.
	By \eqref{eq:steady DP} and H\"{o}lder's inequality, we have
\beq\label{eq:phi kernel control}
\phi=\frac{3}{2c-\phi}\int_{\R}\left[K(x-y)e^{\alpha|x-y|}\right] \frac{\phi^2(y)}{e^{\alpha|x-y|}}\dd y\leq C_{\alpha,p}\left(\int_{\R}\frac{|\phi^2(y)|^{q}}{e^{\alpha q|x-y|}}\dd y \right)^{\frac{1}{q}}.
	\eeq
	Let $l\in [0,\alpha)$ and define
	\beq
		h_{\e}(x):=\frac{e^{l|x|}}{(1+\epsilon e^{|x|})^{\alpha}}\phi(x)
\eeq
for small $\e\in (0,1)$.
Then, for each fixed $\e\in (0,1)$, the function $h_{\e}$ is bounded in $L_{q}(\R)$ by the choice of $l$ and boundedness of $\phi$.  We now prove that $\{h_\e\mid \e\in(0,1)\}$ is uniformly bounded in $ L_{q}(\R)$, which then implies that $\lim_{\e\to 0} h_\e = e^{l|x|}\phi$ belongs to $L_q(\R)$ by dominated convergence and confirms \eqref{eq:expo decay in Lp}.

Since  $\phi$ tends to zero as $|x|\to \infty$, the quadratic nonlinearity guarantees that for every $\delta>0$ there exists  a constant $R_{\delta}>1$ such that
	\begin{equation*}
		|\phi^2(x)|\leq \delta|\phi(x)|\qquad \mbox{for}\quad |x|\geq R_{\delta}.
	\end{equation*}
Since
\begin{equation}\label{spitting h}
\|h_\e\|^q_{L_q(\R)}= \int_\R \left|h_\e(x)\right|^q \dd x \leq C+ \int_{|x|\geq R_\delta} \left|h_\e(x)\right|^q \dd x,
\end{equation}
where $C=C(R_\delta)>0$ is a constant independent of $\e$,
it suffices to study the last integral on the right-hand side of \eqref{spitting h}. 
	
Let $r\in (0,q)$. By \eqref{eq:phi kernel control} and H\"older's inequality, we have 
\begin{align*}
&\int_{|x|\geq R_\delta}|h_{\e}(x)|^{q}\dd x \leq\int_{|x|\geq R_\delta}|h_{\e}(x)|^{q-r}\bigg(\frac{e^{l|x|}}{(1+\epsilon e^{|x|})^{\alpha}}\bigg)^{r}|\phi(x)|^{r}\dd x\\
&\qquad\leq \int_{|x|\geq R_\delta}|h_{\e}(x)|^{q-r}\bigg(\frac{e^{l|x|}}{(1+\epsilon e^{|x|})^{\alpha}}\bigg)^{r} C^{r}_{\alpha,p}\bigg(\int_{\R}\frac{|\phi^2(y)|^{q}}{e^{\alpha q|x-y|}}\dd y\bigg)^{\frac{r}{q}}\dd x\\
&\qquad\leq C^{r}_{\alpha,p}\bigg[\int_{|x|\geq R_\delta}|h_{\e}(x)|^{q}\dd x\bigg]^{\frac{q-r}{q}}\bigg[\int_{|x|\geq R_\delta}\frac{e^{lq|x|}}{(1+\epsilon e^{|x|})^{\alpha q}}\left(\int_{\R} \frac{|\phi^2(y)|^{q}}{e^{\alpha q|x-y|}}\dd y\right)\dd x\bigg]^{\frac{r}{q}}.
\end{align*}
Dividing both sides of the inequality by $\left[\int_{|x|\geq R_\delta}|h_{\e}(x)|^{q}\dd x\right]^{\frac{q-r}{q}}$, we find that\footnote{Note that the term we are dividing by vanishes if and only if $\phi=0$ everywhere in $\{|x|\geq R_\delta\}$, in which case the lemma is obviously true. }
\begin{equation}\label{comb 1}
\int_{|x|\geq R_\delta}|h_{\e}(x)|^{q}\dd x\leq C^{q}_{\alpha,p} \int_{|x|\geq R_\delta}\frac{e^{lq|x|}}{(1+\epsilon e^{|x|})^{\alpha}}\left(\int_{\R} \frac{|\phi^2(y)|^{q}}{e^{\alpha q|x-y|}}\dd y\right)\dd x=:C^{q}_{\alpha,p} T.
\end{equation}
By Fubini's theorem and Lemma \ref{lemma:key estimate for solution decay}, we obtain that
	\begin{align}\label{comb 2}
		\begin{split}
T&=\int_{\R}|\phi^2(y)|^{q}\bigg[\int_{|x|\geq R_\delta}\frac{e^{lq|x|}}{(1+\epsilon e^{|x|})^{\alpha q}e^{\alpha q|x-y|}}\dd x\bigg]\dd y\\
&\leq \int_{|y|\geq R_\delta}|\phi^2(y)|^{q}\frac{B e^{lq|y|}}{(1+\epsilon e^{|y|})^{\alpha q} }\dd y+\int_{|y|<R_\delta}|\phi^2(y)|^{q}\int_{|x|\geq R_\delta}\frac{e^{lq|x|}}{(1+\epsilon e^{|x|})^{\alpha q}e^{\alpha q|x-y|}}\dd x\dd y,
\end{split}
\end{align}
where $B=B(l,q,\alpha)>0$ does not depend on $\e$. Since  $0<l<\alpha $,   the last integral in \eqref{comb 2} is bounded by a constant $C_1$ which depends on $l, \alpha, q, \|\phi\|_\infty$ and $R_{\delta}$ but is independent of $\e$.
Combining \eqref{comb 1}, \eqref{comb 2} and in view  that $|\phi^2(y)|<\delta|\phi(y)|$ for all $|y|\geq R_\delta$, we have
\begin{equation}\label{part 1 bound}
\int_{|x|\geq R_\delta}|h_{\e}(x)|^{q} \dd x\leq C_{\alpha,p}^q \left[\delta ^qB\int_{|x|\geq R_\delta}|h_{\e}(x)|^{q}\dd x+C_1\right] .
\end{equation}
For $\delta$ small enough so that $C_{\alpha,p}^q\delta^q B<\frac{1}{2}$ , \eqref{part 1 bound} implies that
	\begin{equation*}
		\int_{|x|\geq R_\delta}|h_{\e}(x)|^{q}dx\leq C_{2},
	\end{equation*}
	where $C_2=C_{2}(l,\alpha, p, \|\phi\|_\infty,R_{\delta})>0$ is a constant which does not rely on $\e$. 
	
Hence, we have shown that
\[\int_\R |h_{\e}(x)|^{q}dx\lesssim 1.\]
Letting $\e\to 0$, then the dominated convergence theorem ensures that
\begin{equation*}
\int_\R e^{lq |x|}|\phi(x)|^{q}dx\lesssim  1,
\end{equation*}
which implies in particular $x\mapsto e^{l |x|}f(x)\in L_{q}(\R)$ for $q=\frac{p}{p-1}$ and $l\in[0,\alpha)$, and therefore confirms \eqref{eq:expo decay in Lp}. 
	
We now prove that solitary waves decay exponentially by by using \eqref{eq:expo decay in Lp} and Young's inequality in the steady DP equation \eqref{eq:steady DP}:
\beq
e^{\alpha |x|}\phi(x)\lesssim \frac{3}{2c-M}\left[ \left(e^{\alpha |\cdot|K(\cdot)}\right)*\left(e^{\alpha |\cdot|}\phi^2(\cdot)\right)\right](x)\in L^{\infty}(\R)
\eeq
for any $\alpha\in(0,1)$. With this decay estimate, we can use the structure of the DP equation to improve the decay rate to cover the case $\alpha=1$ so that $\phi$ decays at least as good as the kernel $K$. In fact, we have
\[
\begin{split}
e^{|x|}\phi(x)\leq \frac{1}{2c-M}\int_{\R}K(x-y)e^{|x-y|}\left(\phi(y)e^{\frac{|y|}{2}}\right)^2dy\leq \|e^{|\cdot|}K(\cdot)\|_{L^\infty}\|\phi e^{\frac{|\cdot|}{2}}\|_{L^2}^2<\infty.
\end{split}
\]
The above shows that $\phi(x)$ decays as fast as $e^{-|x|}$ at infinity. The fact that the image $e^{|x|}\phi(x)$ forms a simply connected set follows  from the continuity of $\phi$.
\end{proof}

\begin{remark}
This argument  with exponential type convolution estimate is expected to work also for the Whitham and other dispersive equations in proper nonlocal formulation, where a kernel with exponential decay is convoluted with a superlinear nonlinearity.  
\end{remark}

\section{Symmetry and one-crest structure of solitary waves}\label{sect:symmetry and one crest}
With the decay estimates above, we are ready to prove the symmetry for solitary waves to \eqref{eq:DP nonlocal}.    Note that Arnesen recently studied  in \cite{arnesen2018non} the Degasperis-Procesi equation in the nonlocal formulation and proved that traveling waves have wave speed $c$ as the upper bounded. In the following, we will prove that both solitary waves with height smaller than $c$ and solitary waves of the maximum height $c$ are symmetric and have a unique crest, which in particular implies that a peaked solitary wave with decay only has one peak.
For symmetry of solitary waves with height  smaller than the wave speed $c$, our proof follows the idea in \cite{bruell2017symmetry} for the Whitham equation, where the key observation is that the nonlocal operator $L$ behaves as an elliptic operator and there exists a touching lemma on half-plane. This touching lemma  plays the role as the maximum principle for elliptic equations. It is worth to mention that the way to remove the monotonicity assumption on  solitary waves in \cite{bruell2017symmetry} when using the method of moving planes  is inspired by the work of Chen, Li and Ou \cite{CLO} for symmetry of solutions to a class of integral equations induced by fractional Laplacian, although the idea of using Klevin type transform in the latter fails to work in \cite{bruell2017symmetry} due to inhomogeneity of the kernel function for the steady Whitham equation.

\medskip

For solitary waves of the maximum height (see \cite{EW}), i.e., $\sup_{x\in\R}\phi(x)=c$, the argument in \cite{bruell2017symmetry} unfortunately fails to confirm the symmetry. It seems that there exists no argument for confirming the symmetry of a solitary wave with wave speed as its height up to now\footnote{Note that the peaked or cusped waves defined  and found in \cite{lenells2005traveling} are only locally symmetric near a peaked or cusped point.}, so we put forward an argument here for waves of the maximum height and expect it to be effective also for symmetry issues of highest waves of other equations, like the Whitham in \cite{bruell2017symmetry}.
We first introduce the notion of supersolution and subsolution of the steady DP equation.
A solution $\phi$ to the steady Degasperis-Procesi equation \eqref{eq:steady DP} is called a \emph{supersolution} if 
\[\frac{\phi}{3}(2c-\phi)\geq K*\phi^2\]
and a \emph{subsolution} if the inequality above is replaced by $\leq$. With the supersolution and subsolution, we prove the following touching lemma, which can be intuitively explained as: if a supersolution stays above a subsolution on a half plane $(\lambda,\infty)$ for some $\lambda\in\R$, then the supersolution never touches the subsolution at any finite point unless they are equal on the whole half plane $(\lambda,\infty)$ .

\begin{lem}[Touching lemma on a half plane]\label{touching lemma} Let $\phi_1$ and $\phi_2$ be a supersolution and a subsolution of the steady Degasperis-Procesi equation \eqref{eq:steady DP} on a subset  $[\lambda,\infty)\subset \R$, respectively, such that $\phi_1\geq \phi_2$ on  $[\lambda,\infty)$ and $(\phi_1^2-\phi_2^2)(x)=-(\phi_1^2-\phi_2^2)(2\lambda-x)$. Then either
	\begin{itemize}
		\item $\phi_1=\phi_2$ in $[\lambda,\infty)$, or
		\item $\phi_1>\phi_2$ with $\phi_1+\phi_2<2c$ in $(\lambda,\infty)$ .
	\end{itemize}
\end{lem}

\begin{proof}
	In view of its symmetry and monotonicity,  $K$ acts as a positive convolution operator on  functions which is odd with respect to $\lambda$ and does not change sign on the half line $[\lambda,\infty)$. In fact, let $f\geq 0$ on $[\lambda,\infty)$, $f(x)=-f(2\lambda-x)$. Then
	\begin{align*}
		K*f(x)&=\int_{\lambda}^{\infty}K(y)f(x-y) \dd y +\int_{-\infty}^\lambda K(x-y)f(y) \dd y \\
		&=\int_{\lambda}^{\infty}K(x-y)f(y) \dd y +\int^{\infty}_\lambda K(x+y-2\lambda)f(2\lambda-y) \dd y\\
		&= \int_{\lambda}^{\infty}(K(x-y)-K(x+y-2\lambda))f(y)\dd y,
	\end{align*}
	where last equality holds due to $f$ being odd with respect to $\lambda$. 
For $x,y>\lambda$, we have
	\begin{equation}\label{difference between variable}
	(x+y-2\lambda)-|x-y|= 2\min\{x-\lambda,y-\lambda\}> 0
	\end{equation}
Therefore, in view of $K$ being an even function and monotonically decreasing on $(0,\infty)$, we have
	\begin{equation}\label{negative of kernel difference}
	K(x-y)-K(x+y-2\lambda)>0
	\end{equation}
	so that
	\begin{equation*}
		K*f(x)\geq 0 \qquad \mbox{for all }\quad x\geq \lambda.
	\end{equation*}
	In particular, the strict positivity of $K$ implies that either  $K*f>0$ or $f\equiv 0$ on $(\lambda,\infty)$. 
	As a consequence, for the super-and sub-solution $\phi_{1}$ and $\phi_{2}$ in this lemma, we have
	\begin{align*}
		(2c-(\phi_1+\phi_2))(\phi_1-\phi_2)\geq 3K*(\phi_1^2-\phi_2^2)>0
	\end{align*}
	for all $x>\lambda$ unless $\phi_1=\phi_2$ on $[\lambda,\infty)$. The lemma then follows directly.
\end{proof}

We now use the method of moving planes to prove the symmetry and one-crest structure of the wave profile. The first step  is to prove that solitary waves $\phi(x)$ satisfy the following \emph{strict overlay property} in Lemma \ref{able to reflect} below, which means that there exists $\lambda\in\R$ so that for each $x>\lambda$  the reflection  of $\phi(x)$ with respect to $\lambda$ stays strictly above  the value of $\phi$ at the reflection point $2\lambda-x$, i.e., $\phi(x)>\phi(2\lambda-x)$. For convenience, we define the open sets
\[\Sigma_\lambda :=\{x\in \R \mid x> \lambda \}\quad \mbox{and}\quad \Sigma_\lambda ^-:=\{x\in \Sigma_\lambda  \mid \phi(x)<\phi_\lambda(x)\},\]
where $\phi_\lambda(\cdot):= \phi(2\lambda-\cdot)$ is the reflection of $\phi$ about the axis $x=\lambda$.

\begin{lem}[Strict overlay property]\label{able to reflect}
	There exists a $N>0$ suffciently large such that 
	\beq
	\phi(x)>\phi_{\lambda}(x),\quad x>\lambda,
	\eeq
	for any $\lambda \leq-N$. In other words, $\Sigma_\lambda ^-=\emptyset$  for any $\lambda \leq-N$.
\end{lem}
\begin{proof}
	Note that $\phi_{\lambda}(x)$ is also a solution to the steady Degasperis-Procesi equation in nonlocal formulation \eqref{eq:steady DP} if $\phi(x)$ does.  Therefore, we deduce from \eqref{eq:steady DP} that
	\beq\label{eq:diff phi philambda}
	\begin{split}
		&2c(\phi_\lambda(x)-\phi(x))\\
		&=3\left(\int_{\Sigma_{\lambda}\backslash \Sigma_{\lambda}^{-}}+\int_{\Sigma_{\lambda}^{-}}\right)\big(K(x-y)-K(2\lambda -x-y)\big)(\phi_\lambda^{2
		}(y)-\phi^2(y))\dd y \\
		&\quad +\phi_\lambda^{2}(x)-\phi^2(x).
	\end{split}
	\eeq
For  $x\in \Sigma_{\lambda}^{-}$, we use \eqref{difference between variable} and find that the integral over $\Sigma_{\lambda}\backslash\Sigma_{\lambda}^{-}$ on the right side of \eqref{eq:diff phi philambda} is negative so that
	\beq\label{eq: estimate on sigmalambdaminus}
	\begin{split}
		&2c(\phi_\lambda(x)-\phi(x))\\
		&\leq 3\int_{\Sigma_{\lambda}^{-}}\big(K(x-y)-K(2\lambda -x-y)\big)(\phi_\lambda^{2
		}(y)-\phi^2(y))\dd y  +\phi_\lambda^{2}(x)-\phi^2(x).
	\end{split}
	\eeq
	Moreover, Theorem \ref{eq:solution exponential decay} implies that for any small $\epsilon>0$, we can choose suffciently large $N$ such that 
	\beq
	\phi(x)< \phi_{\lambda}(x)<\epsilon, \quad x\in\Sigma_\lambda^-
	\eeq
	for any $\lambda<-N$.
	Then by taking the $L^{\infty}$-norm on both side of \eqref{eq:diff phi philambda} over $\Sigma_\lambda^-$ and using Lemma \ref{touching lemma}, we have
	\begin{equation}\label{lqnorm}
		\begin{split}
			\left\|\phi_\lambda-\phi\right\|_{L^\infty({\Sigma^{-}_{\lambda}})}&\leq \frac{3}{2c}\|\phi+\phi_{\lambda}\|_{L^{\infty}\left({\Sigma_\lambda^-}\right)}\left( \|K\|_{L_{1}({\R})}+1\right)\left\|\phi_\lambda-\phi\right\|_{L_{\infty}({\Sigma^{-}_{\lambda}})}\\
			&\leq \frac{3\epsilon}{c}\left( \|K\|_{L_{1}({\R})}+1\right)\left\|\phi_\lambda-\phi\right\|_{L_{\infty}({\Sigma^{-}_{\lambda}})}\\
		\end{split}
	\end{equation}
	where $\left(\Sigma_\lambda^-\right)^*$ is the reflection of $\Sigma_\lambda^-$ about the plane $x=\lambda$. By choosing $\epsilon<\frac{c}{6(\|K\|_{L_{1}({\R})}+1)}$, we get a contradiction in \eqref{lqnorm} unless $\|\phi-\phi_\lambda\|_{L_\infty({\Sigma^{-}_{\lambda}})}=0$  for $\lambda\leq-N$. As a consequence $\Sigma^{-}_{\lambda}$ must be of measure zero. Since  $\Sigma^{-}_{\lambda}$ is open, we deduce that $\Sigma^{-}_{\lambda}$ is empty for $\lambda\leq-N$.
\end{proof}

\subsection{Solitary waves below the maximum height}\label{sect: symmetry below maximal height}
We are now ready to prove that  solitary waves are symmetric and have exactly one crest at the symmetric axis. The method is similar as that for the Whitham equation in \cite{bruell2017symmetry} but we give full details  for the proof here and write it in a way  to better indicate the obstacle for the case of highest solitary waves.

\begin{thm}
	\label{Symmetry_of_traveling_waves}
Let $\phi$ be a solitary solution to the steady Degasperis-Procesi equation \eqref{eq:steady DP} with $\phi(x)<c$. Then, there exists a unique $\lambda_0\in\R$ such that $\phi$ is symmetric about  $x=\lambda_0$ and $\phi$ is strictly monotonic on each side of the symmetric axis $x=\lambda_0$.
\end{thm}

\begin{proof}
According to Lemma \ref{able to reflect}, there exist $N>0$ such that  $\Sigma_\lambda^-$ is empty for all $\lambda<-N$.
We now move the axis $x=\lambda$ from  $\lambda=-N$ to the right and it is clear that $\Sigma_\lambda^-$ remains empty unless $x=\lambda$ reaches a local maximum of $\phi$, or  there exists $x_0>\lambda$ such that the the reflection image of $\phi$ on the left side of   $x=\lambda$ touches the wave profile on the right side of $x=\lambda$ at  $x_0$, namely $\phi(2\lambda-x_0)=\phi(x_0)$. However, the touching lemma \ref{touching lemma} will exclude the latter case. In fact, if we assume that the latter case happens and the procedure stops at $x=\lambda_0 $ so that $\phi(x)\geq \phi_{\lambda_0}(x)$ with the equality holding for the first time  at $x=x_0>\lambda_0$, but $\phi(x)$ do not match $\phi_{\lambda_0}(x)$ exactly for $x>\lambda_0$. By taking $\phi$ and $\phi_{\lambda}$ as the supersolution and subsolution, respectively, and using Lemma \ref{touching lemma}, we find that $\phi(x)>\phi_{\lambda_0}(x)$ for all $x>\lambda_0$ and a contradiction appears. So, the above process only stops at $x=\lambda_0$, where $\phi$ reaches its local maximum for the first time. 

We now show that $\phi$ is symmetric with respect to $x=\lambda_0$ so that this local maximum of $\phi$ at $\lambda_0$ is just the unique crest. We now assume  $\phi$  be asymmetric with respect to $x=\lambda_0$, and seek a contradiction. First of all, the touching lemma \ref{touching lemma} excludes the possibility for $\phi(x)\equiv\phi(\lambda_0)$ to hold on  $[\lambda_0, \lambda_0+\delta]$ for any small $\delta>0$. Also,  the above process indicates that $\phi$ is strictly increasing on $(-\infty,\lambda_0)$. Then, for any $\epsilon>0$, we  can choose $\delta>0$ sufficiently small such that    $\Sigma_{\lambda}^{-}$  will be simply connected and its size $|\Sigma_{\lambda}^{-}|<\epsilon$ for $\lambda\in (\lambda_0,\lambda_0+\delta)$. For a fixed $\lambda\in (\lambda_0,\lambda_0+\delta)$,  it is clear that  $2\lambda-\lambda_0\in \Sigma_{\lambda}^{-}$. Since  $\phi$ is below the maximum height $c$, we have
\beq
\phi(x)\leq \|\phi\|_{L^{\infty}(\R)}< c
\eeq
for $x\in \Sigma_{\lambda}^{-}$ and
\beq\label{constant c1}
c_{\lambda}:=\sup_{x\in  \Sigma_{\lambda}^{-} }[\phi(x)+\phi_{\lambda}(x)]<2\|\phi\|_{L^{\infty}(\R)}<2c.
\eeq
Then, by simple connectedness of $\Sigma_{\lambda}^{-}$,  we restrict \eqref{eq:diff phi philambda} on $\Sigma_{\lambda}^{-}$ and get
\begin{equation}\label{eq: on sigma lambda one}
(2c-c_{\lambda})(\phi_{\lambda}-\phi)(x)
\leq  3\int_{\Sigma_{\lambda}^{-}}\left[K(x-y)-K(2\lambda -x-y)\right](\phi_\lambda^{2
}(y)-\phi^2(y))\dd y .
\end{equation}
In view of \eqref{constant c1}, we take  $L^{\infty}$-norm over $\Sigma_{\lambda}^{-}$ on both sides of \eqref{eq: on sigma lambda one}, and  get
\begin{equation}\label{contradication below maximum height}
\begin{split}
\|\phi_{\lambda}-\phi\|_{L^{\infty}(\Sigma^{-}_{\lambda})}&< \frac{3|\Sigma_{\lambda}^{-}|}{2c-c_{\lambda}}\|K\|_{L^{\infty}(\Sigma_{\lambda}^{-})}\|\phi_{\lambda}+\phi\|_{L_{\infty}(\Sigma^{-}_{\lambda})}\|\phi_{\lambda}-\phi\|_{L^{\infty}(\Sigma^{-}_{\lambda})}\\
&<\frac{3\epsilon \|\phi\|_{L^{\infty}(\R)} }{c-\|\phi\|_{L^{\infty}(\R)}}\|\phi_{\lambda}-\phi\|_{L^{\infty}(\Sigma^{-}_{\lambda})}\\
\end{split}
\end{equation}
which leads a contradiction if we choose $\epsilon<\frac{c-\|\phi\|_{L^{\infty}(\R)}}{6\|\phi\|_{L^{\infty}(\R)}}$. Therefore $\phi(x)$ matches $\phi_{\lambda_0}(x)$ for all $x\in\Sigma_{\lambda_0}$, i.e., $\phi$ is symmetric with respect to $x=\lambda_0$. In addition, the above process of moving the $x=\lambda$ from far left to $x=\lambda_0$ also guarantees that $\phi$ has a unique crest located at $x=\lambda_0$ and is monotonic on each side of this symmetry axis.	
\end{proof}

\subsection{Solitary waves of the maximum height}\label{sect: symmetry for maximal height} For soliary waves whose crest reaches the maximum height $c$, the term $c_{\lambda}$ can be very close to $2c$  so that $2c-c_{\lambda}$ may be comparable with (or much smaller than) $\epsilon$ and makes \eqref{contradication below maximum height} fail to lead to a contradiction.  In order to get around this difficulty, we have to study the delicate structure of \eqref{eq: on sigma lambda one}. We now explain the idea to get around this difficulty. Suppose that we push $x=\lambda$ from far left  to the right on the real line and the set $\Sigma_{\lambda}^{-}$ remains empty until $x=\lambda$ meet a crest of the wave profile at $\lambda_0$. For $\lambda$ to be slightly larger than $\lambda_0$, the factor $2c-c_{\lambda}$ could be very small but  $|\Sigma_{\lambda}^{-}|$ is also  small. Then for $x,y\in \Sigma_{\lambda}^{-}$, a new and key observation is that the difference  $|2\lambda-x-y|-|x-y|$ satisfies
\begin{equation}\label{eq:distance variable}
 \big||2\lambda-x-y|-|x-y|\big|= 2\min\{x-\lambda,y-\lambda\}\leq 2|\Sigma_{\lambda}^{-}|
\end{equation}
and  is also small.  Therefore, the term $K(x-y)-K(2\lambda -x-y)$ contributes extra smallness which may be used to control the smallness from $2c-c_{\lambda}$.  
\medskip

In the idea above, the size of $2c-c_{\lambda}$ relies on the structure of the wave profile $\phi$ near the crest at $\lambda_0$. 
It is indicated by \cite{lenells2005traveling} and \cite{arnesen2018non} that the wave profile $\phi$ will become non-smooth and a peak or cusp may form at the crest when wave height reach the wave speed $c$. So, it is reasonable to assume that a highest solitary wave $\phi$ is non-smooth at the crest, but we will give a method which work for different non-smooth structures (peak or cusp) near the crest. Without loss of generality, we assume that the crest for the highest solitary wave is located at $x=\lambda_0$ and its local structure is characterized by
\beq\label{eq:structure of peak}
c-\phi(x)\in [C_1|x-\lambda_0|^{\alpha}, C_{2}|x-\lambda_0|^{\alpha}]
\eeq
for $\alpha\in(0,1]$ and some constants $C_1,C_2>0$ when $x$ is very close to $\lambda_0$. In this way,  the argument below can be adapted to treat the symmetry issues of steady solutions  with other H\"{o}lder regularity at the  crest\footnote{It is expected that the type of non-smoothness at the crest of the solitary wave will be the same as that for the convolutional kernel $K(\cdot)$, see the peaked solitary wave for DP in \cite{lenells2005traveling} and  cusped periodic waves for the Whitham in \cite{EW}.}.
\begin{figure}
	\resizebox{12cm}{4cm}{
	\begin{tikzpicture}
	\draw  (-1.5,0) -- (3.9,0);
	\draw [very thin, densely dotted]  (1.2,0) -- (1.2,1.5);
	\draw [very thin, densely dotted]  (1.6,0) -- (1.6,1.131);
	\node at (1.2,-0.3) {${\scriptscriptstyle\lambda}$};
	\draw (-1.5,0.3)to [out=-5,in=230](-0,0.5)to [out=50,in=180](0.8,1.12)to [out=0,in=230](3.1,0.6)to [out=50,in=150](3.9,0.5);
	\draw [densely dashed](1.2,1.05)to [out=20,in=180](1.6,1.12)to [out=0,in=130](2.4,0.5)to [out=310,in=185](3.9,0.3);
	\node at (-1,0.5) {${\scriptscriptstyle\phi}$};
	\draw [very thin, densely dotted] (2.5,0)--(2.5,0.42);
	\draw [very thin, densely dotted] (0.8,0)--(0.8,1.12);
	\node at (0.8,-0.3) {${\scriptscriptstyle \lambda_0}$};
	\node at (2.6,-0.3)  {${\scriptscriptstyle\lambda_1}$};
	\draw[fill] (0.8,0) circle [radius=0.03];
	\draw[fill] (1.6,0) circle [radius=0.03];
	\draw[fill] (2.5,0) circle [radius=0.03];
	\draw[fill] (1.2,0) circle [radius=0.03];
	\node at (1.8,-0.3)  {${\scriptscriptstyle 2\lambda-\lambda_0}$};
	\node at (2.1,0.15) {${\scriptscriptstyle x}$};
	\draw[fill] (2,0) circle [radius=0.03];
	\draw [very thin, densely dotted] (2,0)--(2,0.96);
	\draw [very thin, densely dotted] (2,0.96)--(0.4,0.96);
	\draw [very thin, densely dotted] (0.4,0.96)--(0.4,0);
	\draw[fill] (0.4,0) circle [radius=0.03];
	\node at (0.22,-0.3) {${\scriptscriptstyle 2\lambda-x}$};
	\draw[fill] (0.4,0.96) circle [radius=0.03];
	\node at (-0.1,1.1) {${\scriptscriptstyle \phi(2\lambda-x)}$};
	\draw[fill] (2,0.96) circle [radius=0.03];
	\node at (2.3,1.1) {${\scriptscriptstyle \phi_{\lambda}(x)}$};
	\node at (3.4,0.37) {${\scriptscriptstyle \phi_{\lambda}}$};
	\end{tikzpicture}
}
	\caption{\small{Assume that the procedure of pushing $x=\lambda$ from left to right stops for the first time at a local maximum of $\phi$ at $x=\lambda_0$.  Then, the reflection axis $x=\lambda$ is pushed slightly to the right side of $x=\lambda_0$, which generates a non-empty set $\Sigma_{\lambda}^{-}$ denoted by the  interval $(\lambda,\lambda_1)$. $\phi_{\lambda}$ as the reflection of $\phi$ is partially depicted by dashed lines. For a point $x\in \Sigma_{\lambda}^{-}$, its reflection  $2\lambda-x$ is also dipicted.}}\label{fig:reflection}
\end{figure}
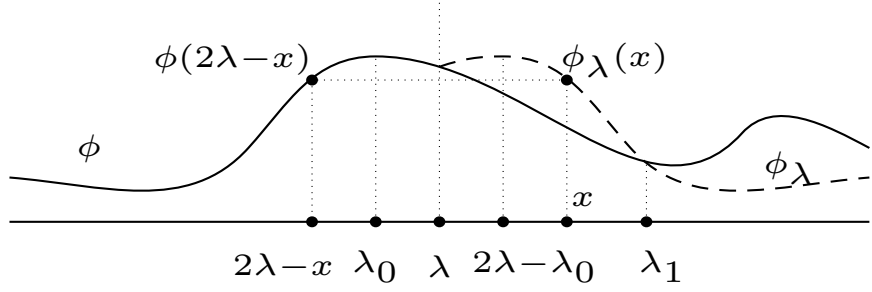
\begin{thm}
	\label{Symmetry_of_traveling_waves II}
There exists a finite $\lambda_0\in\R$ such that the highest solitary solution $\phi$ to the steady Degasperis-Procesi equation is symmetric about $x=\lambda_0$ where the crest is located. Moreover,
 $\phi$ is strictly monotonic on each side of the symmetric axis.
\end{thm}

\begin{proof}
First of all, we  can prove  similarly as in Theorem \ref{Symmetry_of_traveling_waves} that we are able to push $x=\lambda$ from far left to right until it stops at $x=\lambda_0$ where a crest of $\phi$ is located. If $\phi(\lambda_0)<c$, then the proof reduces to the case for waves under the maximum height as in Theorem \ref{Symmetry_of_traveling_waves}. So, we can assume that $\phi(\lambda_0)=c$. As before,  the touching lemma \ref{touching lemma} excludes the possibility for $\phi(x)\equiv\phi(\lambda_0)$ to hold on  $[\lambda_0, \lambda_0+\delta]$ for any small $\delta>0$. Also, for any $\epsilon>0$, we  can choose $\delta>0$ sufficiently small such that    $\Sigma_{\lambda}^{-}$  will be simply connected and its size satisfies $|\Sigma_{\lambda}^{-}|<\epsilon$ for $\lambda\in (\lambda_0,\lambda_0+\delta)$. As in Figure \ref{fig:reflection}, for a fixed  $\lambda\in (\lambda_0,\lambda_0+\delta)$, we denote  $\Sigma_{\lambda}^{-}$ by $(\lambda,\lambda_1)$ and define 
\beq
\delta_1:=\lambda-\lambda_0,\quad \delta_2:=\lambda_1-\lambda.
\eeq
Note that $\delta_2$ can be very small if $\delta_1$ is chosen small enough, and in particular, $\delta_2$ approaches $0$ as $\delta_1$ does. Then, from \eqref{eq:distance variable} and the property of  kernel $K$, we have
\beq\label{eq:diff for K}
0<K(x-y)-K(2\lambda-x-y)\leq 2(x-\lambda),\quad x,y\in \Sigma_{\lambda}^{-}.
\eeq
  The key observation is for estimate of the term $2c-\phi(x)-\phi_{\lambda}(x)$, $x\in\Sigma_{\lambda}^{-}$ as follows: For sufficiently small $\delta_1$ and  any $x\in \Sigma_{\lambda}^{-}$, we use \eqref{eq:structure of peak} and get
\begin{equation}\label{eq: key estimate highest 1}
\begin{split}
2c-(\phi(x)+\phi_{\lambda}(x))&=\phi(\lambda_0)-\phi(x)+\phi(\lambda_0)-\phi(2\lambda-x)\\
&\geq C_1\left[(x-\lambda_0)^{\alpha}+[(2\lambda-\lambda_0)-x]^{\alpha}\right]\\
&\geq C_1|x-\lambda|^{\alpha},
\end{split}
\end{equation}
where the first inequality can be well-illustrated by
\[
c-\phi_{\lambda}(x)=\phi(\lambda_0)-\phi(2\lambda-x)=\phi_{\lambda}(2\lambda-\lambda_0)-\phi_{\lambda}(x)
\]
and the following fact in Figure \ref{fig:reflection} where the distance between $x$ and $2\lambda-\lambda_0$ is the same as the distance between $2\lambda-x$ and $\lambda_0$.
Therefore, for any $x\in\Sigma_{\lambda}^{-}$, we use \eqref{eq:diff for K} and \eqref{eq: key estimate highest 1} to get
\begin{equation}\label{eq: contradiction estimate}
\begin{split}
&\phi_\lambda(x)-\phi(x)\\
&=\frac{3}{2c-(\phi+\phi_\lambda)}\int_{\Sigma_{\lambda}^{-}}\big(K(x-y)-K(2\lambda -x-y)\big)(\phi_\lambda^{2
	}(y)-\phi^2(y))\dd y \\
&\leq \frac{3}{C_1 |x-\lambda|^{\alpha}}\left[2(x-\lambda) |\Sigma_{\lambda}^{-}|\|\phi+\phi_\lambda\|_{L^{\infty}_{\Sigma_{\lambda}^{-}}}\|\phi_\lambda-\phi\|_{L^{\infty}_{\Sigma_{\lambda}^{-}}}\right]\\
&\leq  12cC_{1}^{-1}\delta_2|x-\lambda|^{1-\alpha}\|\phi_\lambda-\phi\|_{L^{\infty}_{\Sigma_{\lambda}^{-}}}.
\end{split}
\end{equation}
From \eqref{eq: contradiction estimate}, we see clearly that $|x-\lambda|^{1-\alpha}$ is a small quantity with non-negative power $1-\alpha$, which shows that the smallness of $K(x-y)-K(2\lambda-x-y)$ balances the singularity caused by the  term $2c-(\phi+\phi_{\lambda})$ on $\Sigma^{-}_{\lambda}$. Then, by choosing $\delta_1$ sufficiently small, we can make $\delta_2<\epsilon<(\frac{C_{1}}{24})^{\frac{1}{2-\alpha}}$ so that
\[
12cC_{1}^{-1}\delta_2|x-\lambda|^{1-\alpha}\leq 12cC_{1}^{-1}\delta_2^{2-\alpha}<\frac{1}{2}.
\] 
Therefore, we get a contradiction by taking the $L^{\infty}_{\Sigma_{\lambda}^{-}}$ norm on the left side of \eqref{eq: contradiction estimate}, and the lemma is proved.
\end{proof}
\begin{remark}
In the proof for Theorem \ref{Symmetry_of_traveling_waves II}, we used the boundedness of the kernel function. For unbounded kernel which may appear in other equations like the Whitham equation, it is expected that proper $L^{p}$-norms instead of $L^{\infty}$-norm should be used for \eqref{eq: contradiction estimate}. 
\end{remark}

\section{A new method for symmetric solutions to be traveling waves}\label{sect: symmetry to traveling}
It has been confirmed in \cite{EHR} that classical symmetric solutions must be traveling waves. The idea for the proof in  \cite{EHR} is to construct a traveling wave solution $\bar{u}(t,x)$ which shares the same initial data with a symmetric solution $u(t,x)$, then the uniquess of solutions implies that $\bar{u}(t,x)$ coincide with $u(t,x)$ so that symmetric solutions are traveling waves. However, we hope to understand how the symmetric structure of waves can be connected with the fixed shape and constant propagation speed, which can not be clearly seen from the constructive proof in \cite{EHR}. With this goal, we check carefully the two constraint conditions (see \eqref{eq:steady form}-\eqref{eq:linear} below) and found that they actually contain information for shape of wave profile and wave propagation speed, respectively. This new finding also leads to a new, more straightforward proof for symmetric solutions to be traveing waves as follows. For convenience, we work on  the Degasperis-Procesi equation in nonlocal formulation \eqref{eq:DP nonlocal}.
\begin{thm}
Solutions  to the Degasperis-Procesi equation with \emph{a priori} spatial symmetry are steady solutions. 
\end{thm}
\begin{proof}
Assume that $u(t,x)$ is a solution to the Degasperis-Procesi equation with symmetric axis $x=\lambda(t)$ for some function $\lambda(\cdot)\in C^{1}(\R)$, i.e.,
\beq\label{eq:sym relation}
u(t,x)=u(t,2\lambda(t)-x).
\eeq
Then, the spatial and time derivatives of $u(t,x)$ satify
\beq\label{eq:derivative sym}
u_{t}|_{(t,x)}=(u_{t}+2\dot{\lambda} u_{x})|_{(t,2\lambda-x)},\quad u_{x}|_{(t,x)}=-u_{x}|_{(t,2\lambda-x)}.
\eeq
In addition, we have
\beq\label{eq:nonlocal sym}
\frac{1}{2}\partial_{x}L(u^2)\big|_{(t,x)}
=-\int_{\R}k(y)[uu_{x}](t,2\lambda-x+y)dy
=-L(uu_{x})\big|_{(t,2\lambda-x)},
\eeq
where in the second equality we used the evenness of the kernel $k(\cdot)$.
Inserting \eqref{eq:sym relation}-\eqref{eq:nonlocal sym} into \eqref{eq:DP nonlocal} and in view of the arbitrariness\footnote{The variable $t$ should of course be chosen from an interval where solutions stay in the same function space as the initial datum does.} of $t$ and $x$, we find that $u$ satisfies the following equation
\beq\label{eq:sym eq}
u_{t}+2\dot{\lambda}u_{x}-uu_{x}-3L(uu_{x})=0,
\eeq
where $\dot{\lambda}:=\dot{\lambda}(t)$ denotes the derivative of $\lambda(t)$ with respect to $t$.
The comparison between \eqref{eq:sym eq} and \eqref{eq:DP nonlocal} then leads to the following constraint conditions
\begin{align}
u_{t}+\dot{\lambda}u_{x}&=0,\label{eq:linear}\\
-\dot{\lambda}u_{x} +uu_{x}+3L(uu_{x})&=0.\label{eq:steady form}
\end{align}
A key observation is that  \eqref{eq:linear} is a linear  PDE of first order with  coefficients relying only on the time variable so that $u(t,x)$ must take the form
\beq\label{eq:ansatz}
u(t,x)=g(x-\lambda(t))
\eeq
for some function $g$, which implies that the shape of the solution will no change in later evolution and the solution propagates with speed $\dot{\lambda}(t)$. Inserting \eqref{eq:ansatz} into \eqref{eq:steady form}, we get the following differential equation
\beq\label{eq:steady form ODE}
\left[-\dot{\lambda}(t)g^{\prime} +gg^{\prime}+3L(gg^{\prime})\right]\Big|_{x-\lambda(t)}=0.
\eeq
Choose arbitrarily two pairs $(t_1,x_1), (t_2,x_2)\in \R^{+}\times \R$ (for which the solution exists and makes sense) such that
\beq
x_1-\lambda(t_1)=x_2-\lambda(t_2)=:X
\eeq
Evaluating \eqref{eq:steady form ODE} at these two pairs gives 
\[
(\dot{\lambda}(t_1)-\dot{\lambda}(t_2))g^{\prime}(X)=0.
\]
Due to the arbitrariness of $X$,   $\dot{\lambda}(t)$ has to be a constant so that the wave profile has a constant propagation speed.  Therefore $u(t,x)$, with fixed shape and constant propagation speed, is a traveling wave solution.
\end{proof}

\section*{Aknowledgement}
The author acknowledges the support by grants No. 231668 and  250070 from the Research Council of Norway and would like to thank  Mats Ehrnstrom for very useful discussions. The author is also grateful to  Mathias Arnesen for explaining his results and to  Luc Molinet for reminding me about peaked solitary waves to the Degasperis-Procesi equation.

%
%
%
%
%

\end{document}